\documentclass[aip,graphicx]{revtex4-1}

\draft 
\usepackage{amsmath,amssymb}
\usepackage{amsthm,graphicx}
\usepackage{color}
\newtheorem{theorem}{Theorem}[section]

\newtheorem{remark}[theorem]{Remark}
\numberwithin{equation}{section}
\newcommand{\R}{\mathbb R}
\begin{document}
	
	
	\title{Controlling Mackey--Glass chaos} 
	
	
	
	\author{ G\'abor Kiss}
 \affiliation{Bolyai Institute, University of Szeged, Szeged H-6720, Hungary}
	\author{Gergely R\"{o}st}
	\email[]{rost@math.u-szeged.hu}	
		\affiliation{Bolyai Institute, University of Szeged,  Szeged H-6720, Hungary, and \\ Wolfson Centre for Mathematical Biology, Mathematical Institute, University of
Oxford, United Kingdom}

	
	\date{\today}

	\date{\empty} 
	

	\begin{abstract}
		The Mackey--Glass equation is the representative example of delay induced chaotic behavior. Here we propose various control mechanisms so that otherwise erratic solutions are forced to converge to the positive equilibrium or to a periodic orbit oscillating around that equilibrium. We take advantage of some recent results of the delay differential literature, when a sufficiently large domain of the phase space has been shown to be attractive and invariant, where the system is governed by monotone delayed feedback and chaos is not possible due to some Poincar\'e--Bendixson type results. We systematically investigate what control mechanisms are suitable to drive the system into such a situation, and prove that constant perturbation, proportional feedback control, Pyragas control and state dependent delay control can all be efficient to control Mackey--Glass chaos with properly chosen control parameters.
		\\
	\end{abstract}
	\maketitle
	\textbf{The Mackey--Glass equation, which was proposed to illustrate nonlinear phenomena in physiological control systems, is a classical example of a simple looking time delay system with very complicated behavior. Here we use a novel approach for chaos control: we prove that with well chosen control parameters, all solutions of the system can be forced into a domain where the feedback is monotone, and by the powerful theory of delay differential equations with monotone feedback we can guarantee that the system is not chaotic any more. We show that this domain decomposition method is applicable with the most common control terms. Furthermore, we propose an other chaos control scheme based on state dependent delays.}
	\section{Introduction}
	\noindent

	The Mackey-Glass equation 
	\begin{equation}
		x'(t)=-\mu x(t)+\frac{p x(t-\tau)}{1+x(t-\tau)^n}, \qquad \mu,p,n,\tau>0, \label{mgeq}
	\end{equation}
	was introduced in 1977 to illustrate some nonlinear phenomena arising in physiological control systems\cite{MG1977}.
	Here $'$ denotes the temporal derivative of a scalar state variable $x(t)$, and the function $f(\xi)=\frac{p\xi}{1+\xi^n}$ represents a feedback mechanism with time delay $\tau$. The interesting situation is $n$ being large when the function $f$ has a distinctive unimodal shape, and in the paper we consider only this case (at least $n>2$). The Mackey--Glass equation provides a benchmark for the application of new techniques for nonlinear delay differential equations as it can generate diverse dynamics, from convergence to oscillations with different characteristics and even chaotic behavior. Despite intensive research over the decades with a number of analytical \cite{rostwu, lrDCDS, gyori}, numerical \cite{farmer,junges,mensour98}, and even experimental studies \cite{amil,kittel}, the emergence of such complexity is not fully understood yet.  
	
	Recent decades showed a growing interest towards chaos control, and several methods have been proposed and applied\cite{handbook}. 
	In this paper we use another strategy,  which we think is novel in the context of chaos control: instead of controlling a particular unstable periodic orbit, we drive all solutions into a domain where the system is governed by monotone feedback.
	
	The delay differential equation
	\begin{equation}
		x'(t)=-\mu x(t)+f(x(t-\tau))
	\end{equation}
	with monotone feedback (where $f'(x)<0$ for all $x$, or $f'(x)>0$ for all $x$) has been widely studied in the mathematical literature and a comprehensive description is available on its global dynamic behaviors for some classes of monotone nonlinearities 
	\cite{krisztin}, 
	and there have been some further interesting new developments as well recently \cite{gabi,moni}.
	One important result is a Poincar\'e--Bendixson type theorem of Mallet-Paret and Sell \cite{mpsell}, which implies that in the case of monotone feedback, bounded solutions converge either to an equilibrium or to a periodic orbit, hence chaotic trajectories are not possible. 
	
	The complexity of the Mackey--Glass equation stems from the combination of time delay and the non-monotonicity of the feedback, and in fact chaotic behavior has been proven for a special class of equations with non-monotone delayed feedback \cite{blw}. A
	domain decomposition method has been proposed
	for unimodal feedback functions\cite{rostwu}, that provides sufficient conditions such that all solutions eventually enter a domain where $f$ is either increasing or decreasing, and in this case the complicated behavior is excluded. In this paper we take advantage of this idea and propose various schemes that can impose such a situation.
	After describing the mathematical background in Section 2, in Section 3 we propose additive control terms, and consider the equation
	\begin{equation}
		x'(t)=-\mu x(t)+\frac{p x(t-\tau)}{1+x(t-\tau)^n}+u(t),\label{mgcontrol}
	\end{equation}
	with control term $u(t)$. We investigate three typical cases, namely constant perturbation $u(t)=k$, proportional feedback control $u(t)=k x(t)$, and the delayed feedback controller $u(t)=k[x(t)-x(t-\tau)]$.  
	We shall say that the chaos is controlled if the system shows complicated behavior for $k=0$, but all solutions eventually enter and remains in some monotone domain of $f$ for some $k\neq 0$, in which case convergence to an equilibrium or to a periodic orbit is guaranteed. In Section 4 we use a different approach: instead of an additive term, we construct a state dependent delay $\tau=\tau(x(t))$ in a proper way so that our domain decomposition method is still applicable. 
	It is important to stress that in this case the form of the controlled equations is of \eqref{mgeq} instead of \eqref{mgcontrol}, and the delay itself will be the subject to the control.
	In Section 5 we illustrate our control mechanisms with a set of numerical simulations, and we conclude the paper with a summary and discussion of the interpretation of our results.
		
	\section{Mathematical background}

	Let $C=C([-\tau,0],\R)$ denote the Banach space of continuous
	functions $\phi: [-\tau,0] \to \R$ with the usual sup norm
	$||\phi||=\max_{-\tau \leq s \leq 0} |\phi(s)|.$
	Given its biological interpretation, traditionally only non-negative solutions of \eqref{mgeq} are studied, hence we restrict our attention the cone
	$$C_+=\{\phi \in C : \phi(s) \geq 0, -\tau \leq s \leq 0 \},$$
	and define the corresponding order intervals
	$$[\phi,\psi]:=\{\zeta \in C: \psi -\zeta \in C_+, \zeta- \phi \in C_+\}.$$ 

	Every $\phi \in C_+$ determines a unique continuous function
	$x=x^\phi: [-\tau,\infty) \to \R$, which is differentiable on
	$(0,\infty)$, satisfies \eqref{mgeq} for all $t>0$, and $x(s)=\phi(s)$
	for all $s\in [-\tau,0]$. It is easy to see that the cone $C_+$ is positively invariant, i.e. a
	solution $x^\phi(t)$ with non-negative initial function $\phi$
	remains non-negative for all $t \geq 0$. Existence and uniqueness extend to \eqref{mgcontrol} too when $u(t)$ has the usually required smoothness, however non-negativity should be checked in each specific case. The segment $x_t
	\in C$ of a solution is defined by the relation $x_t(s)=x(t+s)$,
	where $s \in [-\tau,0]$ and $t\geq 0$, thus $x_0=\phi$ and
	$x_t(0)=x(t)$, and the family of maps
	$$\Phi: [0,\infty)\times C_+ \ni (t,\phi) {\color{red}\mapsto} x_t(\phi):=\Phi_t(\phi)  \in C_+$$ defines a continuous semiflow on $C_+$. For any $\xi \in \R$, we write $\xi_*$
	for the element of $C$ satisfying $\xi_*(s)=\xi$ for all $s \in
	[-\tau,0]$. The equilibria $\xi_*$ of \eqref{mgeq} are given by the solutions
	of $\mu \xi = f(\xi)$.
	The trivial equilibrium is $0_*$, and in addition there exists at most one
	positive equilibrium $K_*$ given by $K=(p/\mu-1)^{1/n}$. Note that $f'(\xi)=p(1-(n-1)\xi^n)(1+\xi^n)^{-2}$,
	so $f'(0)=p$ and there is a unique $\xi_0=(n-1)^{-1/n} $ such that $f'(\xi_0)=0$. The function $f$ is increasing on $[0,\xi_0]$, have its maximum $f(\xi_0)=p(n-1)^{1-1/n}n^{-1}$, and decreasing on $[\xi_0,\infty)$
	with $\lim_{x\to \infty} f(x)=0$.
	Depending on the parameters, there are three fundamental situations:\\
	\textit{(a)} if $\mu \geq p$ then only the zero equilibrium exists; \\
	\textit{(b)} if $\mu <p \leq \mu (1+(n-1)^{-1})$ then there is a positive equilibrium $K_*$ on the increasing part of $f$ (i.e. $K\leq \xi_0)$; \\
	\textit{(c)} if $p > \mu (1+(n-1)^{-1})$ then there is a positive equilibrium $K_*$  on the increasing part of $f$ (i.e $K>\xi_0$ or equivalently $\mu<f(\xi_0)/\xi_0$). \\
	It is well known \cite{rostwu} that in case \textit{(a)} all solutions converge to $0$ and in case \textit{(b)} all positive solutions converge to $K$, regardless of the delay. Thus here we consider only the interesting case \textit{(c)}, when the following numbers

	$$
	\beta:=\frac{f(\xi_0)}{\mu} , \quad \quad
	\alpha:=\frac{f(\beta)}{\mu}=\frac{f(\frac{f(\xi_0)}{\mu})}{\mu} $$  also play a crucial role in
	characterizing the nonlinear dynamics of equation \eqref{mgeq}. A cornerstone of this paper is the following result, which combines Theorem 3.5 (R\"ost \& Wu \cite{rostwu})   and Theorem 8 (Liz \& R\"ost \cite{lrDCDS}), ensuring that the long term dynamics is governed by a monotone part of the feedback function.
	
	\begin{theorem}\label{alap}
		Let $g(x)=\mu^{-1} f(x)$, and assume $g'(0)>1$ and $K>\xi_0$. Then, if either condition
		$$
		g^2(\xi_0)>\xi_0 \eqno{(L)}
		$$
		or 
		$$
		h^2(\xi_0)>\xi_0, \text{ where } h(x)=(1-e^{-\mu \tau})g(x)+e^{-\mu \tau}K  \eqno{(T)}
		$$
		holds, then every solution eventually enters and remains in the domain where $f'$ is negative, hence converging to $K$ or to a periodic solution oscillating around $K$.
	\end{theorem}
	The assumption of this theorem means that we are in case \textit{(c)}. Then the interval $[\alpha_*,\beta_*]$ is attractive and invariant \cite{rostwu}, and condition $(L)$ means $\alpha>\xi_0$. Results relating attractive invariant intervals of the discrete map $f$ to attractive invariant intervals for \eqref{mgeq} originate from Ivanov \& Sharkovsky \cite{IS}, and recently have been successfully used for other problems as well \cite{lizruiz,xingfu}.
	Note that this condition is independent of $\tau$, hence in this situation chaotic behavior can not appear by increasing the delay. The delay dependent condition $(T)$ is built on earlier works \cite{gyori,liztrofim}.

	\section{Controlling Mackey-Glass chaos with additive terms}
	
	Our aim is to choose our additive control term $u(t)$ from three common classes, in a way that some analogue of Theorem \ref{alap} holds for \eqref{mgcontrol}.
	
	\subsection{Constant perturbation control}
	For any $k\in\mathbb R$, we consider
	\begin{equation}\label{eq:mgconstc}
		x'(t)=-\mu x(t)+\frac{p x(t-\tau)}{1+x(t-\tau)^n}+ k.
	\end{equation}
	\begin{theorem}\label{constpert}
		Assume that $K>\xi$ but $(L)$ is not satisfied, that is, $g^2(\xi_0)\leq\xi_0$ in \eqref{mgeq}. Then the following statements hold:
		
		(i) there is a $k_*<\mu \xi_0$ such that for all $k\geq k_*$, \eqref{eq:mgconstc} has no complicated solution;
		
		(ii) there is an explicitly computable $k_1$, such that for $k<k_1$, \eqref{eq:mgconstc} has no equilibria and solutions become unfeasible;
		
		(iii) for $k_1<k<k_2:=\mu \xi_0-f(\xi_0)$, there are two positive equilibria $K_1$ and $K_2$, and solutions with initial function $\phi \in [(K_1+k/\mu)_*,\xi_{0*}]$ converge to $K_2$; 
		
		(iv) there exists a $k_3$ such that for $k_2<k<k_3$, 
		\eqref{eq:mgconstc} has no complicated solutions.
		
	\end{theorem}
	\begin{proof}
		After using the change of variable $y=x-\frac{k}{\mu}$, \eqref{eq:mgconstc} reads as
		$$
		y'(t)=-\mu y(t)+p\frac{ y(t-\tau)+\frac{k}{\mu}}{1+\left(y(t-\tau)+\frac{k}{\mu}\right)^n}
		$$
		That is
		\begin{equation}\label{eq:y}
			y'(t)=-\mu y(t)+f_k(y(t-\tau))
		\end{equation}
		with $f_k(\xi)=f\left(\xi+\frac{k}{\mu}\right)$, thus adding the constant perturbation $k$ has the same effect as shifting the graph of $f$ by $k/\mu$. Note that we are interested only in non-negative solutions $x(t)$ of \eqref{eq:mgconstc}, that is $y(t)\geq -k/\mu$, and we call such solutions feasible.  Let $\hat\xi_0=\xi_0-\frac{k}{\mu}$, $
		\hat\beta=\frac{f_k(\hat\xi_0)}{\mu} ,
		\hat\alpha=\frac{f_k(\hat\beta)}{\mu}.$
		Clearly,  $f_k'(\hat\xi_0)=0$ and $\hat\beta=\beta$, that is, $\hat\alpha=f_k(f(\xi_0))$. 
		
		(i) For $k>0$, the graph of $f$ is shifted to the left, solutions remain positive and we also have $\hat \alpha>0$, with $\liminf_{t\to \infty} y(t)\geq  \hat \alpha$ (analogously to Theorem 3.5 from R\"{o}st \& Wu\cite{rostwu}). At $k\geq \mu \xi_0$, $\hat \xi_0\leq 0<\hat \alpha$, and by continuity the relation  $\hat \xi_0<\hat \alpha$ must hold on some interval $(k_*,\mu\xi_0)$ as well.
		
		(ii) We shift the graph of $f$ to the right until equilibria are destroyed. In the critical case $k=k_1$, $f$ is tangential to $\mu \xi$, so first we find the unique $\xi_\mu>0$ such that $\mu=f'(\xi_\mu)=\frac{p(1+(1-n)\xi_\mu^n)}{(1+\xi_\mu^n)^2}$. This is a quadratic equation in $\xi_\mu^n$, and taking its positive root we find
		$$\xi_\mu=\left(\frac{-2\mu-p(n-1)+\sqrt{4p\mu n+p^2(n-1)^2}}{2\mu}\right)^{\frac{1}{n}}.$$
		When the graph is shifted by $k_1/\mu$, the tangent line of the shifted graph $f_{k_1}$ with slope $\mu$ is exactly the line $\mu \xi$, so we must have $f(\xi_\mu)=\mu (\xi_\mu-k_1/\mu)$,
		which gives $k_1=f(\xi_\mu)-\mu \xi_\mu$.
		For $k<k_1$, $f_k(\xi)<\mu \xi$ holds on $[-k/\mu,\infty)$, where $f_k$ is defined. For a solution $y(t)$, let
		$v(t):=y(t)+\int_{t-\tau}^t f_k(y(s))ds$, then $v'(t)=-\mu y(t)+f(y(t))<-\min_{\xi\geq -k/\mu}[\mu \xi-f_k(\xi)]<0. $ This means that $v(t)$ becomes smaller than $-k/\mu$ in finite time, but due to $y(t)<v(t)$, each solution $y(t)$ becomes unfeasible.  
		
		(iii) For $k_1<k<0$, there are always two equilibria of \eqref{eq:y}, now we are looking for an other critical value $k_2$ that separates the cases when the larger equilibrium is on the decreasing part of $f_k$ from when both are on the increasing part. The critical case is characterized by one of the equilibria being $\hat \xi_0$, that is $\mu \hat \xi_0=f_k(\hat \xi_0)=f(\xi_0)$, and $k_2=\mu(\xi_0-f(\xi_0))$ follows. For $k_1<k<k_2$ there are two equilibria, $\hat K_1$ and $\hat K_2<\hat \xi_0$. It is easy to see that there are initial functions $\phi$ with $\phi(0)=-k/\mu$, $\phi(\theta)$ small for $\theta<0$ such that the derivative of the solution is negative at zero, thus unfeasible solutions exist.
		To avoid such situations, we restrict our attention to the interval $[\hat K_{1*},\hat \xi_{0*}]$, where $f_k$ is monotone increasing. For solutions with segments from this interval,  $y(t)=\hat K_1$ implies $y'(t)\geq-\mu K_1+f_k(\hat K_1)=0$, and $y(t)=\hat \xi_0$ implies $y'(t)\leq -\mu \hat \xi_0+f_k(\hat \xi_0)<0$, therefore this interval is invariant. Now we can apply Proposition 2 from R\"{o}st \& Wu\cite{rostwu} to show that all solutions in this interval converge to $\hat K_2$. Transforming back to variable $x$ we obtain (iii).
		
		(iv) First notice that $f_{k_2}(f_{k_2}(\hat \xi_0)/\mu)=\mu \hat \xi_0.$ Our goal is to show that $$D(k):=f_{k}(f_{k}(\hat \xi_0)/\mu)-\mu \hat \xi_0=f_{k}(f( \xi_0)/\mu)-\mu  \xi_0+k=f((f( \xi_0)+k)/\mu)-\mu  \xi_0+k >0$$ in an interval $(k_2,k_3)$, then an analogue of Theorem \ref{alap} provides the result. Differentiating  with respect to $k$ gives
		$$D'(k)=f'((f( \xi_0)+k)/\mu)/\mu+1, $$
		and evaluating at $k_2=\mu \xi_0-f(\xi_0)$ we arrive at
		$$D(k_2)=0, \quad D'(k_2)=f'(\xi_0)/\mu+1=1>0. $$
		
	\end{proof}

	\subsection{Proportional feedback control}
	
	In this subsection, we consider $u(t)=kx(t)$.
	The rearrangement
	\begin{equation}\label{eq:prop} x'(t)=-(\mu-k) x(t)+\frac{p x(t-\tau)}{1+x(t-\tau)^n}\end{equation}
	shows that the control has no effect on key properties of the nonlinearity in \eqref{mgeq}.
	
	With $w=\mu-k$, Theorem \ref{alap} can be directly applied.
	\begin{theorem} \label{propthm} Assume $\alpha<\xi_0<K$. Then the following holds:
		
		(i) there is a $k_*<0$, such that for $k\in (\mu-f(\xi_0)/\xi_0,k_*)$, 
		\eqref{eq:prop} has no complicated solutions; 
		
		(ii) {\color{red}i}f $\mu-p<k<\mu-f(\xi_0)/\xi_0$, then all solutions converge to  $(p/(\mu-k)-1)^{1/n}$; 
		
		(iii) if $k\leq \mu-p$ then all solutions converge to 0; 
		
		(iv) if $k>\mu$, all solutions converge to infinity.
	\end{theorem}
	
	\begin{proof} (i)
		For a given $k$, let
		\begin{equation}
			\tilde\beta=\frac{f(\xi_0)}{\mu-k}=\frac{p(n-1)^{\frac{n-1}{n}}}{n(\mu-k)},\qquad
			\tilde\alpha=\frac{f(\tilde\beta)}{\mu-k}=\frac{p^2(n-1)^{\frac{n-1}{n}}n^{n}(\mu-k)^{n}}{n(\mu-k)^2\left(n^n(\mu-k)^n+p^n(n-1)^{n-1}\right)},
		\end{equation}
		and $\tilde g = f/(\mu-k)$.
		Notice that  $k={\mu-f(\xi_0)}/{\xi_0}$ means that $\tilde \alpha=\tilde \beta$, and $\tilde g^2(\xi_0)=\xi_0.$
		Hence, to apply $(L)$ to \eqref{eq:prop}, we want to show that
		$$
		\tilde\alpha\frac{1}{\xi_0}=\frac{p^2(n-1)n^{n-1}(\mu-k)^{n-2}}{\left(n^n(\mu-k)^n+p^n(n-1)^{n-1}\right)}>1
		$$
		for some $k$. For simplicity we write $w=\mu-k$, and let
		$$
		S(w)=\frac{p^2(n-1)n^{n-1}w^{n-2}}{\left(n^nw^n+p^n(n-1)^{n-1}\right)},
		$$
		and $w_0=\frac{p(n-1)}{n}$. 
		It is easy to check that $S(w_0)=1$.
		Furthermore,
		$$
		S'(w)=p^2(n-1)^{\frac{n-1}{n}}n^{n}\frac{p^n(n-1)^{n-1}(n-2)w^{n+1}-2n^nw^{2n+1}}{nw^4\left(n^nw^n+p^n(n-1)^{n-1}\right)^2},
		$$
		hence
		$$
		S'(w_0)=p^2(n-1)^{\frac{n-1}{n}}n^{n}\frac{\left(\frac{p(n-1)}{n}\right)^{n+1}(n-2n)}{nw^4\left(n^nw^n+p^n(n-1)^{n-1}\right)^2}<0
		$$
		and
		$S'(\hat w)=0$ only for $\hat w=\frac{p(n-1)}{n} \sqrt[n]{\frac{n-2}{2n-2}} <w_0.$
		These, together with the facts $S(0)=0$ and
		$	S'(w)>0	$ 
		for $w\in\left(0,\hat w\right)$, imply the existence of a unique $w_*<w_0$ satisfying $S(w_*)=1$.
		That is for $(\mu-k)\in(w_*,w_0)$, every solution enters the interval $[\tilde\alpha,\tilde\beta]$, where $f$ is monotonically decreasing prohibiting the existence of chaotic solutions. Shifting back, $(\mu-k)\in(w_*,w_0)$ is equivalent with $k \in (\mu-w_0,\mu-w_*)$, and with $k_*=\mu-w_*$ and noting that $w_0=f(\xi_0)/\xi_0$, we conclude (i). To see (ii) and (iii), notice that in these cases \eqref{eq:prop} falls in the case of (b) and (c) as described in Section 2, thus Proposition 3.2 and Proposition 3.1 from R\"{o}st \& Wu\cite{rostwu} give the result.  To check (iv), from $x'(t)>(k-\mu)x(t)$ , convergence to infinity is clear for $k>\mu$.
	\end{proof}
	
	Next we give a delay dependent result.
	
	\begin{theorem} \label{propdelaythm} Assume that $K>\xi_0$ and $(L)$ does not hold for $\tilde g$ with some $k$. Then for sufficiently small delay, $(T)$ holds for $\tilde h$. Furthermore, the smaller the delay, the larger the range of $k$  that enables chaos control.
	\end{theorem}
	\begin{proof} The first statement is obvious, since as $\tau \to 0$, $(T)$ becomes $K>\xi_0$ regardless of $k$. For the second statement, note that the control parameter does not change $\xi_0$, but $K$ becomes $\tilde K$. Fix all the parameters but $\tau$ such that $\tilde K>\xi_0$, let $w=\mu-k$ and denote by $\tilde h_\tau$ the function in condition $(T)$ corresponding to equation \eqref{eq:prop}, belonging to a given $\tau$. We show that if $\tau_1<\tau_2$, then $\tilde h^2_{\tau_1}(\xi_0)> \tilde h^2_{\tau_2}(\xi_0)$.  Since $\tilde g(\xi_0) > \tilde K$,  we have $\tilde h_{\tau_2}(\xi_0)>\tilde h_{\tau_1}(\xi_0)>\tilde K$, and for $\xi>\tilde K$, $\tilde g(x)<\tilde K$ implies $\tilde h_{\tau_2}(\xi)<\tilde h_{\tau_1}(\xi)$. Together with the monotone decreasing property of $\tilde h$ for $\xi>\tilde K$, we find
		$$\tilde h^2_{\tau_1}(\xi_0)=\tilde h_{\tau_1}(\tilde h_{\tau_1}(\xi_0))>\tilde h_{\tau_2}(\tilde h_{\tau_1}(\xi_0))  >\tilde h_{\tau_2}(\tilde h_{\tau_2}(\xi_0))=  \tilde h^2_{\tau_2}(\xi_0).$$
		The conclusion is that for $\tau_1<\tau_2$, if $\tilde h^2_{\tau_2}(\xi_0)>\xi_0$ holds, then  $\tilde h^2_{\tau_1}(\xi_0)>\xi_0$ also holds, thus if $k$ is a good control for some delay (in the sense that $(T)$ holds), it is a good control for all smaller delays as well. The consequence is that for smaller delays we always have a larger range of $k$ such that $(T)$ still holds.
		
	\end{proof}

	\subsection{Pyragas control}
	A popular control mode  is $u(t)=k(x(t-\tau))-x(t)$, and with such term \eqref{mgeq} becomes
	
	\begin{equation}
		x'(t)=-(\mu+k) x(t)+\frac{p x(t-\tau)}{1+x(t-\tau)^n}+ kx(t-\tau),\nonumber
	\end{equation}
	that is
	\begin{equation}
		x'(t)=-(\mu+k) x(t)+F_k(x(t-\tau))\label{eq:pyr}
	\end{equation}
	with $F_k(\xi)=f(\xi)+k\xi$. Notice that while the Pyragas control changes the shape of the nonlineariy, it does not change the equilibria of the system.

	\begin{theorem}\label{pyrthm}
		Assume $K>\xi_0$ and $g^2(\xi_0)<\xi_0$. Then for $k>\frac{p(n-1)^2}{4n}$, all solutions of\eqref{eq:pyr} converge to $K$.
		
	\end{theorem}
	
	\begin{proof}
		(i) A straightforward calculation shows that the function $f'(\xi)=\frac{p(1-(n-1)\xi^n)}{(1+\xi^n)^{2} }$ has a minimum when $\xi^n=\frac{n+1}{n-1}:$ let $b(u)=\frac{p(1-(n-1)u)}{(1+u)^{2} },$ then $b'(u)=\frac{p (n (u-1)-u-1)}{(u+1)^3},$ and $b'(u)=0$ exactly at $u=(n+1)/(n-1)$. Therefore $f'(\xi)\geq \frac{-n p}{\left(\frac{n+1}{n-1}+1\right)^2}=\frac{-p(n-1)^2}{4n},$ with equality at that point. Hence if $k>\frac{p(n-1)^2}{4n}$, then $F_k'(\xi)=f'(\xi)+k>0$, and in this case \eqref{eq:pyr} is governed by positive monotone feedback. Since $F_k(\xi)<(\mu+k) \xi$ for $\xi>K$, it is easy to see that any $[0_*, L_*]$ interval is invariant whenever $L>K$, and the same proof as Proposition 3.2. in R\"{o}st \& Wu\cite{rostwu} ensures that all positive solutions converge to $K$.
		
	\end{proof}
	
	When $k<0$, then there is a $\breve \xi$ such that $F_k(\breve \xi)<0$, and then solutions with initial functions satisfying $\phi(0)=0$ and $\phi(-\tau)=\breve \xi$ immediately become negative. Since the non-negative cone is not invariant any more, here we don't discuss Pyragas control with negative $k$. 
	
	When $0<k<\frac{p(n-1)^2}{4n}$ then $F_k(\xi)$ has a bimodal shape, with local extrema $q_1<q_2$. The numbers $q_1$ and $q_2$ can be found as the solutions of
	$f'(\xi)=-k$, which is quadratic in $\xi^n$, so it is possible to find them explicitly, similarly to case (ii) of Theorem \ref{constpert}. It is natural to try to apply an analogue of $(L)$ condition in the bimodal case too, and forcing all solutions into the domain $(q_1,q_2)$, where $F$ is monotone decreasing. Nevertheless, the required conditions $q_1<\alpha_k$ and  $\beta_k<q_2$ become analytically intractable, and one can find parameter settings when they fail when $k$ being near either zero or $\frac{p(n-1)^2}{4n}$. An other possibility is to forcing solutions to the increasing part of $F_k$ thus expecting convergence to $K$ again, so we may require $\alpha_k>q_2$, that is $F_k(F_k(q_1)/(\mu+k))>(\mu+k)q_2$, but again that seems too involved to find a simply interpretable condition.

	\section{State dependent delay control}
	
	From Theorem \ref{alap} is is clear that chaos can be controlled by decreasing the delay to a small quantity, since as $\tau\to 0$, condition $(T)$ becomes $K>\xi_0$, hence for sufficiently small $\tau$, $(T)$ is satisfied. However, it may be impossible or very expensive to permanently keep $\tau$ small, thus here we explore how can we establish chaos control when we modify the delay only temporarily, depending on the current state. Thus, we consider equation
	\begin{equation}
		x'(t)=-\mu x(t)+\frac{p x(t-r(x(t)))}{1+x(t-r(x(t)))^n} \label{sdde}
	\end{equation}
	with state dependent delay $r(x(t))$, where one can interpret $r(x(t))=\tau-k(x(t))$ with baseline delay $\tau$ and delay control $k(x(t))$. It is reasonable to assume $k(x(t))\geq 0 $ and $k(x(t))<\tau$, then $r(x(t))\in (0,\tau]$. We say that a solution is slowly oscillatory, if $x(t)-K$ has at most one sign change on each time interval of length $\tau$.

	\begin{theorem}\label{sddethm}
		Assume $K>\xi_0$ and let $\hat K<\xi_0$ be defined by $f(\hat K)=f(K)$. Let
		$\tau_*:=\min\{\tau, \frac{K-\xi_0}{f(\xi_0)},\frac{\xi_0-\hat K}{f(\xi_0)}\},$ and
		$\zeta=(\tau-\tau_*)(f(\xi_0)-\mu \xi_0)$.
		
		Define the following state dependent delay function:\\
		$r(x)=\tau$ for $x\geq\xi_0+\zeta$;\\
		$r(x)=\tau_*$ for $x \leq \xi_0$;\\
		$r(x)$ is $C^2$-smooth and monotone on $[\xi_0,\xi_0+\zeta]$ with  $r'(x)\leq (f(\xi_0)-\mu\xi_0)^{-1}$.
		
		Then, solutions of \eqref{sdde} eventually enter the domain where $f'$ is negative and slowly oscillatory complicated solutions can not exist.
	\end{theorem}
	
	\begin{proof}
		
		The existence and uniqueness of solutions have been discussed in Krisztin and Arino\cite{arino}. Since $\tau_*\leq r(x(t))\leq \tau$, we can deduce that $[\alpha_*,\beta_*]$ is attractive and invariant analogously to the constant delay case Theorem 3.5. in R\"{o}st \& Wu\cite{rostwu}. For solutions in this interval, $|x'(t)|<f(\xi_0)$ holds. Now we claim that positive solutions always go beyond $\xi_0$, i.e. $\limsup_{t \to \infty} x(t) > \xi_0$. Assume the contrary, then there is a solution $x(t)>0$ such that $x(t)<\xi_0+\epsilon$ holds for all $t>t_0$ with some $0<\epsilon<K-\xi_0$. Define $$z(t)=x(t)+\int_{t-r(x(t))}^t f(x(s))ds.$$ Then $z(t)<\xi_0+\epsilon+\tau f(\xi_0)$, but
		$z'(t)=-\mu x(t)+f(x(t))> \min_{\xi \in [\alpha,\xi_0]} (f(\xi)-\mu\xi)>0$ for all $t>0$, which is a contradiction.
		Hence for any positive solution there is a $t_*$ such that $x(t_*)>\xi_0$.  Next we show that for all $t\geq t_*$, $x(t)>\xi_0$ also holds. Assuming the contrary, there exists a $t^*$ such that
		$x(t^*)=\xi_0$ and $x'(t^*)\leq 0$. 
		Note that
		$$\xi_0=x(t^*)=x(t^*-r(\xi_0))+\int_{t^*-r(\xi_0)}^{t^*} x'(s) ds>x(t^*-r(\xi_0))-r(\xi_0) f(\xi_0),$$
		so 
		$x(t-r(\xi_0))<\xi_0+r(\xi_0) f(\xi_0)<K$.  Similarly, 
		$x(t-r(\xi_0))>\xi_0-r(\xi_0) f(\xi_0)>\hat K.$
		But then  $x'(t^*)=-\mu \xi_0 +f(x(t-r(\xi_0)))\geq \mu (K-\xi_0)>0$, a contradiction.
		We conclude that solutions enter the domain where $f'<0$ and remain there.
		To apply the Poincar\'e--Bendixson type results of Krisztin--Arino \cite{arino}, we need to confirm the increasing property of $t\mapsto t -r(x(t))$, cf. condition (H2) of \cite{arino}. This is equivalent to $r'(x)x'(t)<1$, which obviously holds outside $(\xi_0,\xi_0+\zeta)$. Within $(\xi_0,\xi_0+\zeta)$, $x'(t)<f(\xi_0)-\mu\xi_0$ is valid, hence one can find a $C^2$-smooth $r(x)$ such that $r(\xi_0)=\tau_*$, $r(\xi_0+\zeta)=\tau$, and meanwhile $r'(x)\leq (f(\xi_0)-\mu\xi_0)^{-1}$.
		
		Then we can apply Theorem 8.1. of Krisztin \& Arino\cite{arino}, and thus slowly oscillatory solutions converge to $K$ or to a periodic orbit.
	\end{proof}
	
	\begin{remark}
		Some recent results of Kennedy \cite{Kennedy}, that have not been published yet, suggest that Theorem IV.1. can be extended from slowly oscillatory solutions to all solutions.
	\end{remark}
	
	While the control scheme in this theorem may seem complicated, what it really means is that when a solution is approaching $\xi_0$ from above, we decrease the delay in a way that the solution will turn back before reaching $\xi_0$, hence forcing it to stay in the domain where $f'<0$. In particular, $k(x)=0$ for $x\geq \xi_0+\zeta$, $k(x)=\tau-\tau_*$ for $x \leq \xi_0$ and some intermediate control $k(x)$ applied when the solution is in the interval $(\xi_0,\xi_0+\zeta)$. For such equation with state dependent delay the Poincar\'e--Bendixson type theorem was proven only to the subset of slowly oscillatory solutions, hence at the current state-of-the-art of the theory we can not say more, but the applicability of this control scheme will be illustrated in the next Section.

	\section{Applications, simulations and discussion}
	
	We investigated a number of possible mechanisms so that with a well chosen control parameter, an otherwise chaotic Mackey--Glass system is forced to show regular behavior.
	The Mackey--Glass equation was used to model the rate of change of circulating red blood cells, and most of our results have a meaningful interpretation in this context. For example, $u(t)=k$ with $k>0$ may represent medical replacement of blood cells at a constant rate, or $u(t)=kx(t)$ with negative $k$ may represent increased destruction rate of blood cells which can be achieved by administration of antibodies \cite{foley}. Our approach is different from typical chaos control methods, since our strategy is to choose a control such that all solutions will be attracted to a domain where the feedback function is monotone, and then some Poincar\'e--Bendixson type results exclude the possibility of chaotic behavior. By applying this domain decomposition method, which is based on R\"{o}st \& Wu \cite{rostwu}, instead of stabilizing a particular orbit, we push the full dynamics into a non-chaotic regime.

	For $u(t)=k$, clearly $k>0$ helps the cell population, and as Theorem \ref{constpert} shows, with sufficiently large $k$ chaos can always be controlled regardless of the delay.
	A somewhat counterintuitive part of Theorem \ref{constpert} is that for some negative $k$ it is possible to force the system to converge to a positive equilibrium, however one has to be careful as the system will collapse if $k$ is below the threshold $k_1$. This is illustrated in Fig. 1, where on the left we can see how $k>0$ controls a chaotic solution into a periodic one, while in the right we can observe that decreasing $k<0$ first regulate the system into periodic behavior, then to convergence, and finally to collapse (i.e. hitting zero in finite time).  
	
	The proportional feedback control $u(t)=k x(t)$ again helps the population when $k>0$ and when it fully compensates the baseline mortality ($\kappa>\mu$), the population grows and unbounded (Theorem \ref{propthm}, (iv)). Yet, controlling chaos is best achieved with $k<0$, when the destruction of cells is increased, then with a fine tuning of $k$ the dynamics can be made regular (Theorem \ref{propthm}, (i) and (ii)), which is shown on the left panel of Fig. 2. If cell destruction is too high, the population goes extinct (Theorem \ref{propthm}, (iii)). Theorem \ref{propdelaythm} gives a delay dependent result, showing that even if the condition $(L)$ fails, chaos control can be achieved by satisfying $(T)$. We showed that the smaller the delay, the easier the control is, in the sense that we can pick $k$ from a larger range to satisfy $(T)$.
	On the right panel of Fig.2., we illustrated this delay dependent feature: when we switched on the control we decreased the delay temporarily to show that with this smaller delay it is a good control, but when we reset the delay at some time later, the delay dependent condition $(T)$ fails and the solution goes back to irregular mode with the same control.
	
	We also used the popular Pyragas control $u(t)=k(x(t-\tau)-x(t))$. The conclusion of our Theorem
	\ref{pyrthm} is that for positive $k$, the unimodal shape of the nonlinearity turns into a bimodal shape, and when $k$ is large enough (our theorem explicitly tells us how large), then the 
	nonlinearity is transformed into a monotone feedback, as the control term overwhelms the original unimodality. Once we achieved monotonicity, we can use results from R\"{o}st and Wu\cite{rostwu} to prove that solutions converge to the positive equilibrium. Figure 3 left shows how such regulation occurs as we increase $k$. For negative $k$ the non-negative cone is not invariant anymore so we do not consider this possibility. Let us remark that control of Mackey--Glass chaos has been experimentally observed with Pyragas-type control \cite{kittel}, and here our results give an analytic explanation how and why this happens.
	
	Finally, we considered a very different type of control, taking advantage of some results from the theory of state dependent delays. While it is clear from Theorem \eqref{alap} that chaos can be eliminated when the delay is sufficiently small, in Theorem \eqref{sddethm} we constructed a state dependent delay function, that allows us to construct a delay control scheme where the delay is reduced only in a part of the phase space. This is illustrated in the right panel of Figure 3, where we applied delay reduction only in the region $x<K$, and that was sufficient to drive the irregular solution into periodic behavior.

	
	
	%
	%
	
	%
	
	\begin{acknowledgments}
GK was supported by ERC Starting Grant Nr. 259559 and the EU-funded Hungarian grant EFOP-3.6.1-16-2016-00008. GR was supported by OTKA K109782 and Marie Sk\l odowska-Curie grant agreement No 748193.
	\end{acknowledgments}
	

	\begin{figure}
		\includegraphics[scale=0.63]{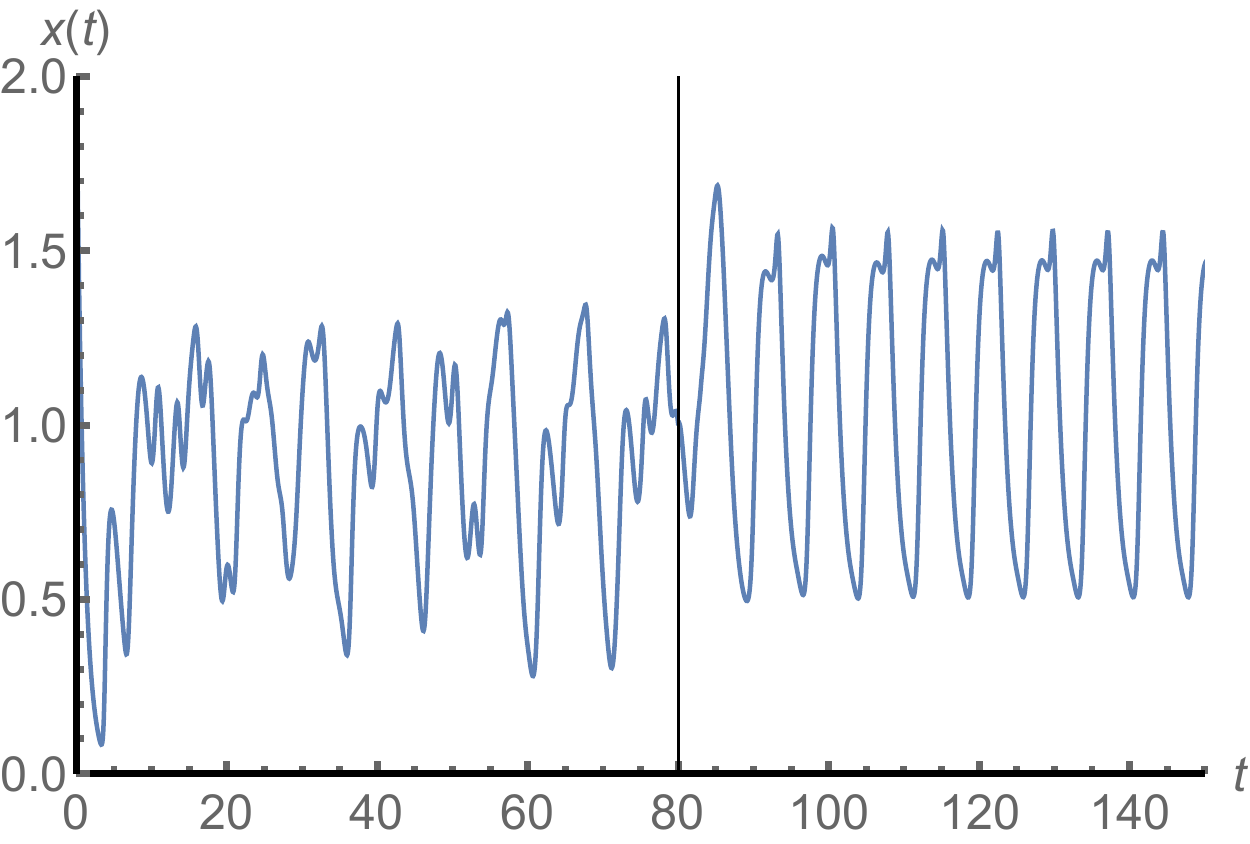}\quad\includegraphics[scale=0.63]{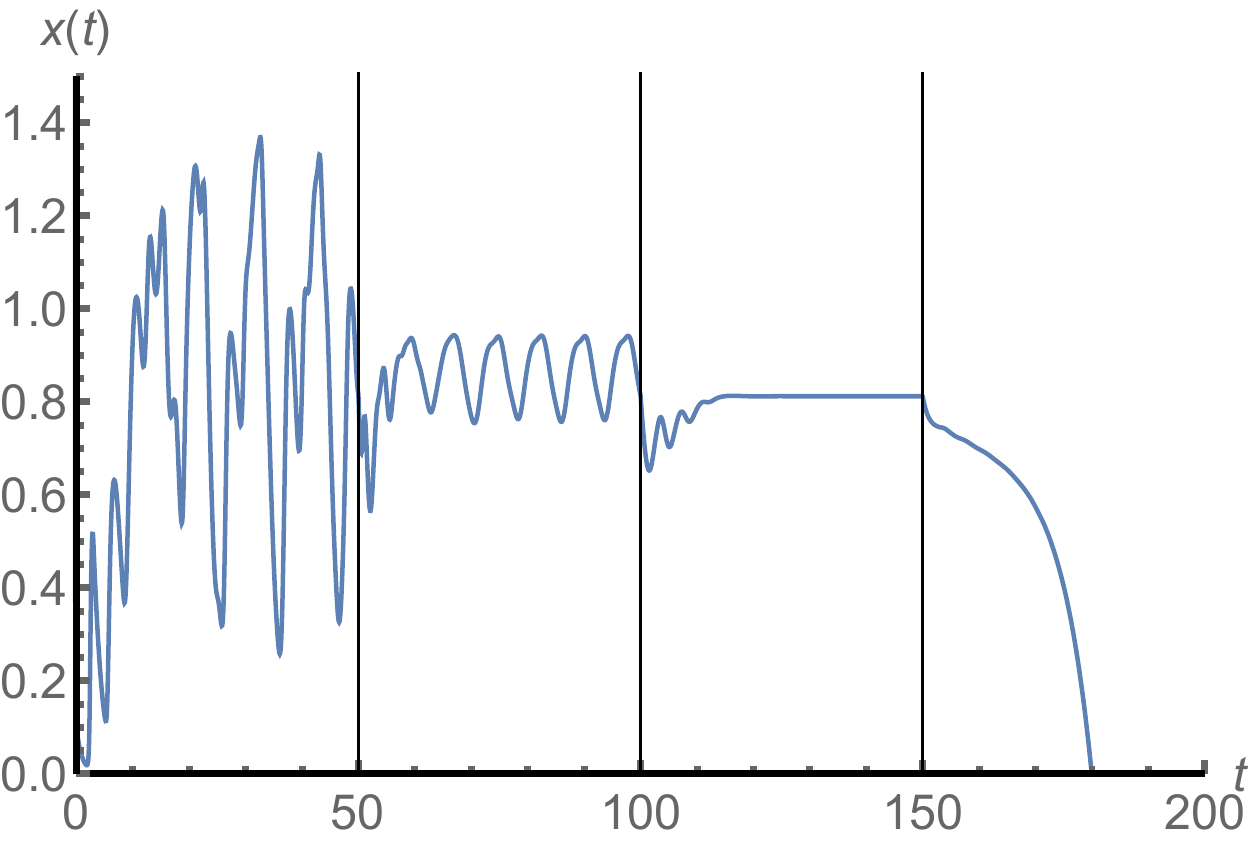}
		\caption{\textbf{ Constant perturbation control: illustrations to Theorem \ref{constpert}}. On the left a numerical solution to \eqref{eq:mgconstc} is plotted. For $0\leq t<80$, there is no control ($k=0$), and the solution is irregular. At $t=80$, we switch on the constant control with $k=0.39$. The initial function was $2+0.02\sin t$.
			On the right, for $0\leq t<50$, $k=0$, and the solution is irregular. At $t=50$, $k$ was decreased to $-0.48$, and the solution becomes periodic. At $t=100$, $k$ is set to $-0.62$, and the solution converges to $K$. From $t=150$, we use $k=-0.69$ and the solution reaches $0$ in finite time. The initial function was $-1.2t+0.1e^t$. In both cases, the parameters were set to $\mu=1,\ \tau=3,\ p=2,\ n=9.65$.} 
	\end{figure}
	
	\begin{figure}
		\includegraphics[height=5.2cm]{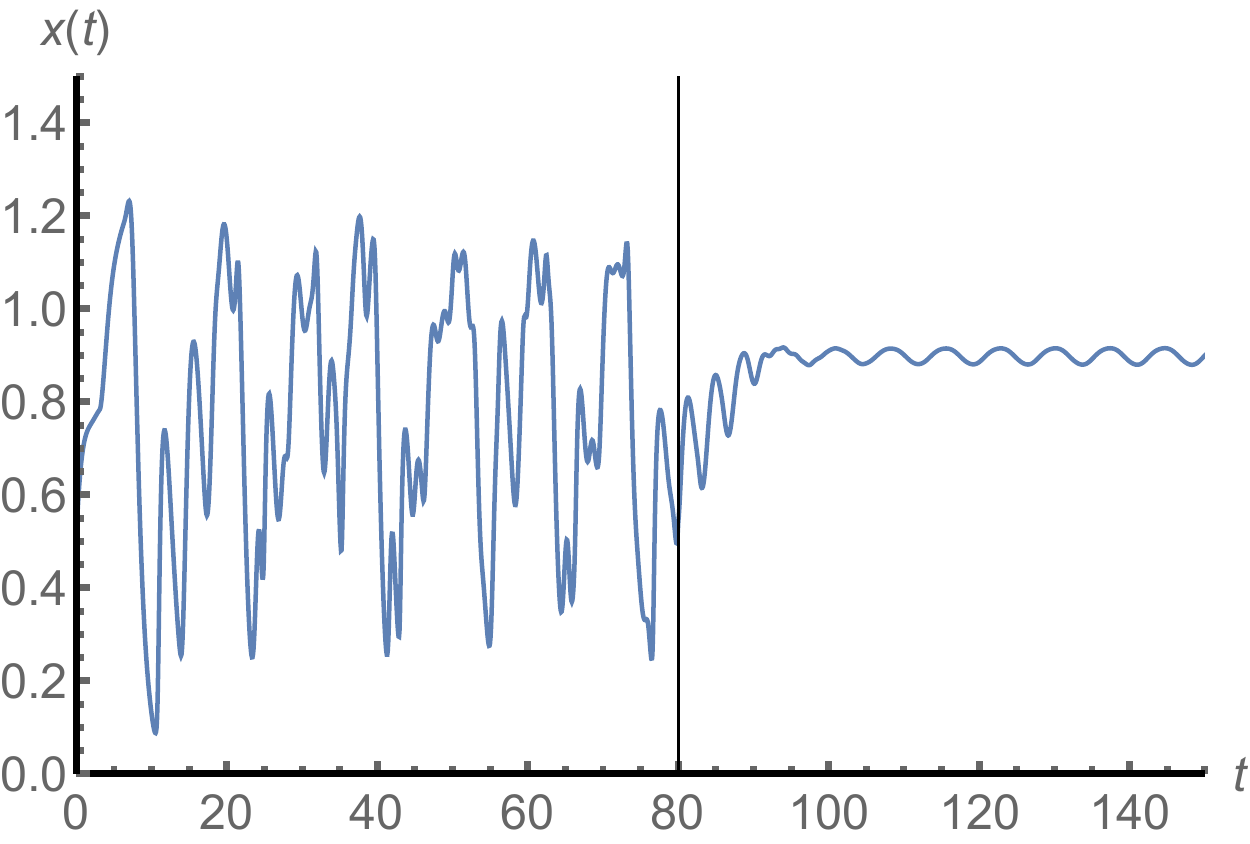}\quad\includegraphics[height=5.2cm]{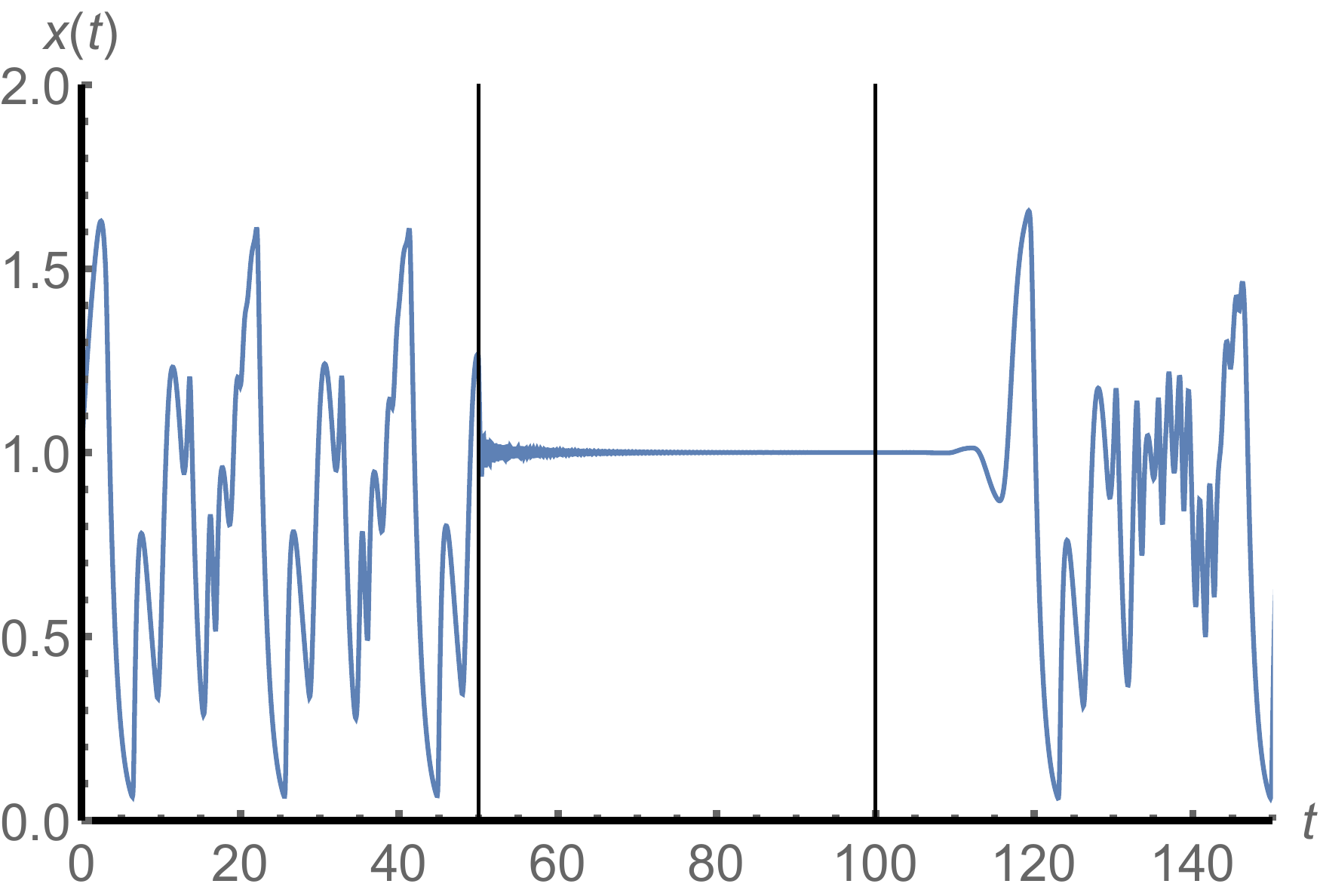}
		\caption{\textbf{Proportional feedback control: illustrations to Theorems \ref{propthm} and \ref{propdelaythm}}. Numerical solutions to \eqref{eq:prop} are plotted. On the left for $0\leq t<80$, there is no control ($k=0$), and the solution is irregular. At $t=80$, we switch on the proportional control with $k=-0.507$. The other parameters were $n=20,\mu=1.275,\tau=3.11,p=2$. With these parameters, the condition in $(i)$ of Theorem \ref{propthm} is satisfied, and the solution converges to a regular oscillation. The initial function was $0.5+0.01\cos 2t$. 
			On the right, for $0\leq t<50$, $\tau=3$ and $k=0$. At $t=50$, to illustrate Theorem \ref{propdelaythm}, $\tau$ is decreased to $0.125$ and $k$ is set to $-0.022$, so $(T)$ holds and the solution behaves regularly. From $t=100$, $\tau=3$, and while $k$ is still $-0.022$, now $(T)$ fails with this larger delay and the solution becomes irregular again. The other parameters were  $n=27.9,\ \mu=0.97,\ p=2$. The initial function was $1+0.1t$.
		} 
	\end{figure}

	\begin{figure}
		\includegraphics[height=5.2cm]{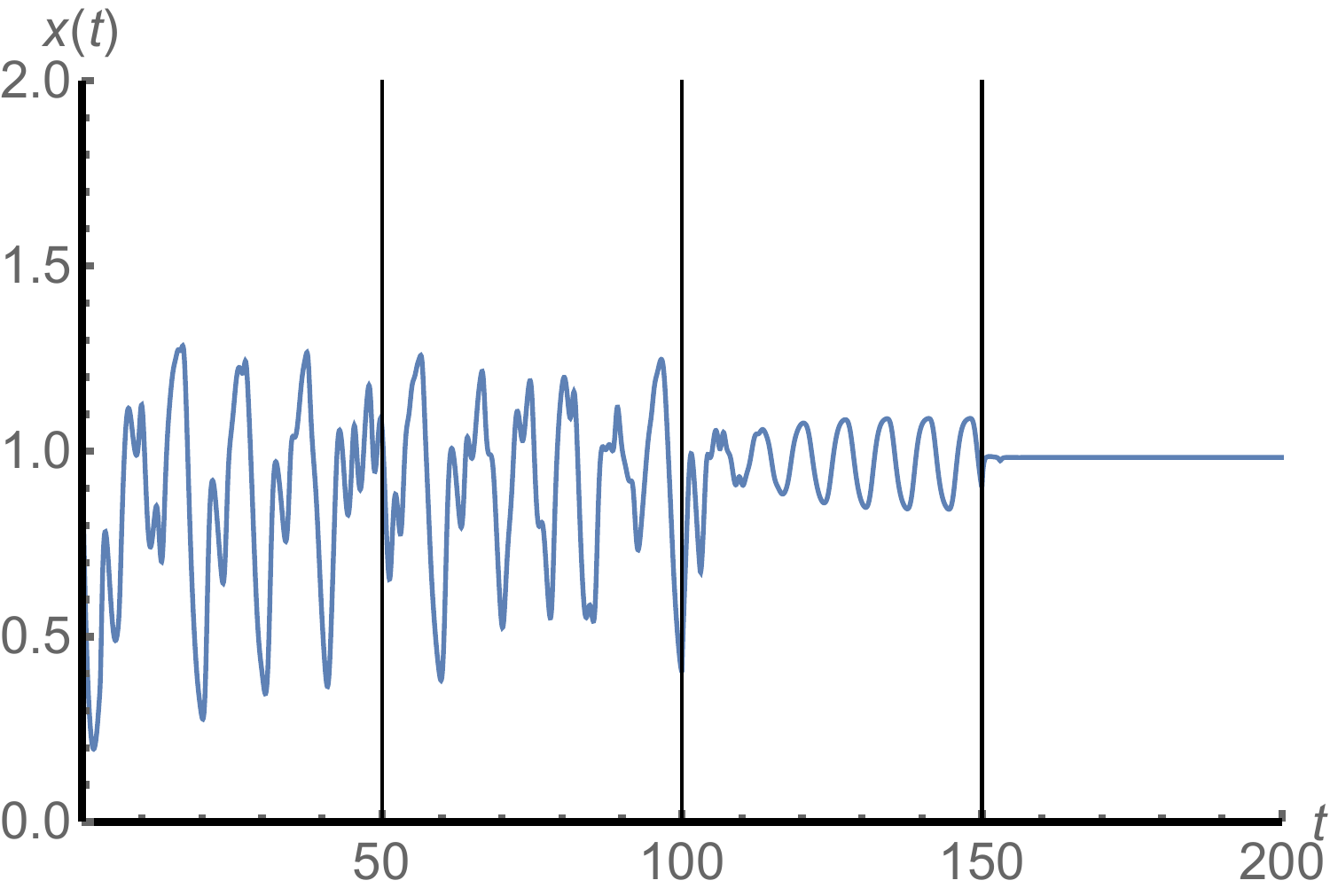}\quad\includegraphics[height=5.2cm]{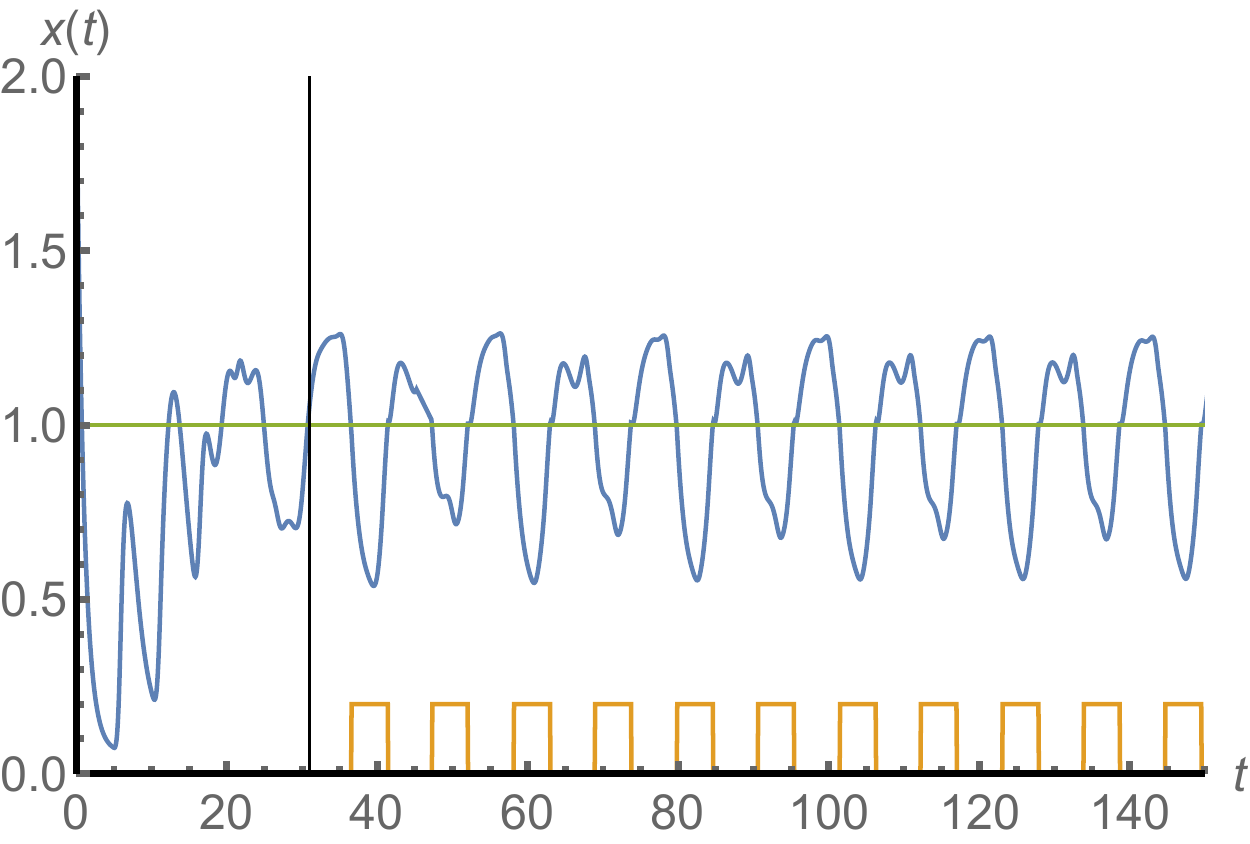}
		\caption{\textbf{Left: Pyragas control  illustration to Theorem \ref{pyrthm}.} Numerical solution to \eqref{eq:pyr} is plotted. For $0\leq t<50$, there is no control ($k=0$), and the solution is irregular. At $t=50$, we switch on the Pyragas control with $k=0.08$ which is too small to cease the irregularity of the solution, which becomes periodic after $t=100$ when the control was increased to $k=0.95$. Finally, the solution converges $K$ after $t=150$ when the control increased further to $k=3.9$, when the condition of Theorem \ref{pyrthm} holds. The other parameters were set to $\mu=1.08,\ \tau=3,\ p=2,\ n=9.65$. The initial function was $1+0.1e^{-t}$. \textbf{Right: State dependent delay control.} A numerical solution to \eqref{sdde} is plotted. We switch on the delay function scheme at $t=31$, which drives the solution to a periodic orbit. The horizontal line shows the equilibrium and it is also the boundary for delay reduction, where for the sake of simplicity we used a step function for $r(x(t))$: the delay is $5$ for $x>K$ and $4$  for $x<K$. In the lower part of the graph it is shown when the delay control is on or off. Parameter values are $n=6$, $p=2$, $\mu=1$, initial function is $2_*$. 
		} 
	\end{figure}%

\begin{thebibliography}{0}
		
		\bibitem{amil} Amil, P., Cabeza, C., Masoller, C. and Martí, A.C., 2015. Organization and identification of solutions in the time-delayed Mackey-Glass model. Chaos: An Interdisciplinary Journal of Nonlinear Science, 25(4), p.043112.
		
		\bibitem{farmer} Farmer, J. D. (1982). Chaotic attractors of an infinite-dimensional dynamical system. Physica D: Nonlinear Phenomena, 4(3), 366-393.
		
		\bibitem{foley} Foley, C., and Mackey, M.C.: Dynamic hematological disease: a review. J. Math. Biol. 58, 285–322 (2009)
		
		\bibitem{gyori}  Gy\H{o}ri, I. and Trofimchuk, S. I. 1999 Global attractivity in $x'(t)=-\delta x(t)+p f (x(t-\tau))$. Dynam. Syst.
		Appl. 2, 197--210.
		
		\bibitem{xingfu} Huang, C., Yang, Z., Yi, T., and Zou, X. (2014). On the basins of attraction for a class of delay differential equations with non-monotone bistable nonlinearities. Journal of Differential Equations, 256(7), 2101-2114.
		
		\bibitem{ivanovliz}  Ivanov A. F., Liz  E. and Trofimchuk  S., Global stability of a class of scalar nonlinear delay differential equations, Differential Equations Dynam. Systems, 11, 33-54, 2003.
		
		\bibitem{IS} Ivanov, A. F., and Sharkovsky, A. N. (1992). Oscillations in singularly perturbed delay equations. In Dynamics Reported (pp. 164-224). Springer Berlin Heidelberg.
		
		\bibitem{junges} Junges, L. and Gallas, J.A., 2012. Intricate routes to chaos in the Mackey–Glass delayed feedback system. Physics letters A, 376(30), pp.2109-2116.
		
		\bibitem{Kennedy} Kennedy B., The Poincar\'{e}--Bendixson Theorem For a Class of Delay Equations with State-Dependent Delay and Monotonic Feedback, in preparation
		
		\bibitem{kittel} Kittel, A. and Popp, M., 2008. Application of a Black Box Strategy to Control Chaos. Handbook of Chaos Control, Second Edition, pp.575-590.
		
		\bibitem{krisztin} Krisztin, T., 2008. Global dynamics of delay differential equations. Periodica Mathematica Hungarica, 56(1), pp.83-95.
		
		\bibitem{arino} Krisztin, T. and Arino, O., 2001. The two-dimensional attractor of a differential equation with state-dependent delay. Journal of dynamics and differential equations, 13(3), pp.453-522.
		
		\bibitem{moni} Krisztin, T., Polner, M. and Vas, G., 2016. Periodic Solutions and Hydra Effect for Delay Differential Equations with Nonincreasing Feedback. Qualitative Theory of Dynamical Systems, pp.1-24.
		
		\bibitem{gabi} Krisztin, T. and Vas, G., 2016. The unstable set of a periodic orbit for delayed positive feedback. Journal of Dynamics and Differential Equations, 28(3-4), pp.805-855.
		
		\bibitem{Kuang}  Kuang Y., Delay Differential Equations with Applications in Population Dynamics, vol. 191 of Mathematics in Science and Engineering, Academic Press, Boston, Mass, USA, 1993.
		
		\bibitem{blw} Lani-Wayda B.  and Walther H.-O. , Chaotic motion generated by delayed negative
		feedback. Prt I: A transversality criterion, Differential and Integral Equations 8 (1995),
		1407–1452.
		
		
		\bibitem{lrDCDS}  Liz E., and  R\"{o}st G., On global attractors for delay differential equations with unimodal feedback,
		Discrete and Continuous Dynamical Systems, 24:(4) 1215-1224 (2009)
		
		\bibitem{lizruiz} Liz, E., and Ruiz-Herrera, A. (2015). Delayed population models with Allee effects and exploitation. Math. Biosci. Eng, 12, 83-97.
		
		\bibitem{liztrofim}  Liz E.,  Tkachenko V. and  Trofimchuk S., A global stability criterion for scalar functional
		differential equations, SIAM J. Math. Anal., 35 (2003), 596–622.
		
		\bibitem{MG1977} Mackey, M. C. and Glass, L. (1977). Oscillation and chaos in physiological control systems. Science, 197(4300):287-289.
		
		\bibitem{mpsell} Mallet-Paret J.  and  Sell G., The Poincar\'e--Bendixson theorem for monotone cyclic
		feedback systems with delay, J. Differential Equations 125 (1996), 441–489.
		
		\bibitem{mensour98} Mensour, B. and Longtin, A., 1998. Chaos control in multistable delay-differential equations and their singular limit maps. Physical Review E, 58(1), p.410. 
		
		\bibitem{pyr1} Namajūnas, A., Pyragas, K. and Tamaševičius, A., 1995. Stabilization of an unstable steady state in a Mackey-Glass system. Physics Letters A, 204(3-4), pp.255-262.
		
		\bibitem{pyr2} Pyragas, K., 2006. Delayed feedback control of chaos. Philosophical Transactions of the Royal Society of London A: Mathematical, Physical and Engineering Sciences, 364(1846), pp.2309-2334.
		
		\bibitem{rostwu} R\"{o}st, G. and Wu, J., Domain-decomposition method for the global dynamics of delay differential equations with unimodal feedback, Proc. R. Soc. A, 463, 2655-2669, 2007.
		
		\bibitem{handbook} Sch\"oll, E. and Schuster, H.G. eds., 2008. Handbook of chaos control. John Wiley \& Sons.
		
		
		
		
	\end{thebibliography}
\end{document}